\documentclass[12pt]{amsart}
\usepackage{amsfonts, amsmath, latexsym, epsfig}
\usepackage{amssymb}
\usepackage{url}

\newcommand{\RR}{\ensuremath{\mathbb{R}}}

\newcommand{\QQ}{\ensuremath{\mathbb{Q}}}

\newcommand{\ZZ}{\ensuremath{\mathbb{Z}}}

\newtheorem{theorem}{Theorem}[section]

\newtheorem{lemma}[theorem]{Lemma}

\theoremstyle{definition}
\newtheorem{remark}[theorem]{Remark}

\def\QuotS#1#2{\leavevmode\kern-.0em\raise.2ex\hbox{$#1$}\kern-.1em/\kern-.1em\lower.25ex\hbox{$#2$}}

\DeclareMathOperator{\Aut}{Aut}
\DeclareMathOperator{\Sym}{Sym}

\DeclareMathOperator{\Isom}{Isom}

\begin{document}

\author{Mathieu Dutour Sikiri\'c}
\address{Mathieu Dutour Sikiri\'c, Rudjer Boskovi\'c Institute, Bijenicka 54, 10000 Zagreb, Croatia}
\email{mdsikir@irb.hr}

\author{Anna Felikson}
\address{Anna Felikson, Independent University of Moscow, B. Vlassievskii 11, 119002 Moscow, Russia}
\email{felikson@mccme.ru}

\author{Pavel Tumarkin}
\address{Pavel Tumarkin, Independent University of Moscow, B. Vlassievskii 11, 119002 Moscow, Russia}
\email{pasha@mccme.ru}

\thanks{Research supported by the
Croatian Ministry of Science, Education and Sport under
contract 098-0982705-2707 (M.D.), RFBR grant 07-01-00390-a (A.F. and P.T.), and INTAS grants YSF-06-10000014-5916 (A.F.) and YSF-06-10000014-5766 (P.T.)}

\title{Automorphism groups of root systems matroids}

\date{}

\begin{abstract}
Given a root system $\mathsf{R}$, the vector
system $\tilde{\mathsf{R}}$ is obtained by taking a representative $v$
in each antipodal pair $\{v, -v\}$. The matroid $M(\mathsf{R})$
is formed by all independent subsets of $\tilde{\mathsf{R}}$.
The automorphism group of a matroid is the group of permutations preserving
its independent subsets.
We prove that the automorphism groups of all irreducible root systems
matroids $M(\mathsf{R})$  are uniquely determined
by their independent sets of size $3$. As a corollary, we compute 
these groups explicitly, and thus complete the classification of 
the automorphism groups of root systems matroids.
\end{abstract}

\maketitle

\section{Introduction}

Given a vector $v\in \RR^n$, denote by $H_{v}$
the hyperplane of vectors orthogonal to $v$ and
by $s_v$ the orthogonal reflection along $H_v$.
A {\em root system} $\mathsf{R}$ is a finite
family of vectors $v\in \RR^n$, such that:
\begin{itemize}
\item $\mathsf{R}\cap \RR v=\{v, -v\}$ for all $v\in \mathsf{R}$,
\item $s_v \mathsf{R}=\mathsf{R}$ for all $v\in \mathsf{R}$.
\end{itemize}
The norms of the roots are not specified a priori in our definition.
If $\mathsf{R}$ splits into $r$ orbits under the action of $W(\mathsf{R})$, 
then $r$ norms, a priori different, are possible. 
A root system $\mathsf{R}$ is {\em irreducible} if
$\mathsf{R}$ cannot be decomposed into two
orthogonal components.

The groups $W(\mathsf{R})$ generated by the reflections
$(s_v)_{v\in \mathsf{R}}$ are exactly finite Coxeter groups. We call a finite  
Coxeter group {\em indecomposable} if the corresponding root system is irreducible. 
Finite indecomposable Coxeter groups are classified into
the following ones:  $\mathsf{A}_n$, $\mathsf{B}_n$,
$\mathsf{D}_n$, $\mathsf{E}_6$, $\mathsf{E}_7$, $\mathsf{E}_8$,
$\mathsf{F}_4$, $\mathsf{I}_2(m)$, $\mathsf{H}_3$ and $\mathsf{H}_4$
(see e.g.~\cite[Chapter 2]{humphreyscoxeter}).



Given a finite set $X$, a {\em matroid} $M$ is a
family ${\mathcal I}$ of subsets $S$ of $X$ (called {\em independent sets}) such that:
\begin{itemize}
\item ${\mathcal I}\not= \emptyset$
\item for any $S\in {\mathcal I}$, any $S'\subset S$ one has 
$S'\in {\mathcal I}$.
\item If $A, B\in {\mathcal I}$, $|A| > |B|$ then $\exists x\in A\setminus B$ such that $B\cup\{x\}\in {\mathcal I}$.
\end{itemize}
One way to get a matroid is to take $X$ to be a family of
vectors and ${\mathcal I}$ the linearly independent subsets
of $X$.
We denote by ${\mathcal B}$ the set of all {\em bases}, i.e. maximal
independent sets of a matroid.
A {\em circuit} is a non-independent set such that each of its
proper subset is independent; we denote by ${\mathcal C}$,
respectively ${\mathcal C}_3$, the set of circuits of a matroid, respectively
circuits of order $3$.
A matroid is uniquely defined either by its independent
sets, bases or circuits.
The matroid automorphism group $\Aut(M)$ is the set of all
permutations of $X$, which preserve ${\mathcal I}$, or
equivalently ${\mathcal B}$ or ${\mathcal C}$.

Given a root system $\mathsf{R}$, $X=\tilde{\mathsf{R}}$ is obtained
by selecting a representative in each pair $\{v, -v\}$
of vectors.
We define a matroid $M(\mathsf{R})$ on $X$ by
taking ${\mathcal I}$ to be the subsets of $X$
that are linearly independent.




We prove the following theorem.

\begin{theorem}\label{m3}
Let $\mathsf{R}$ be an irreducible root system.
A permutation $\phi$ of $\tilde{\mathsf{R}}$ is an automorphism of $M(R)$ if and only if it
preserves ${\mathcal C}_3$.

\end{theorem}

The automorphism groups $\Aut(M(\mathsf{R}))$ for classic root systems $\mathsf{A}_n$, $\mathsf{B}_n$ and $\mathsf{D}_n$ were computed in~\cite{cubesimplexcases}.
The root system $\mathsf{F}_4$, respectively $\mathsf{H}_3$ was investigated in~\cite{F4syst}, respectively \cite{ehly}.
While proving Theorem~\ref{m3}, we compute also the automorphism groups 
of the root system matroids for the remaining exceptional root systems, so we complete the classification of the groups $\Aut(M(\mathsf{R}))$.    

Denote by ${\Isom}(\mathsf{R})$ the group of isometries of $\RR^n$ preserving $\mathsf{R}$ (see Section~\ref{isom} for our choice of root lengths). 
Clearly, any isometry of $\mathsf{R}$ is an element of $\Aut(M(\mathsf{R}))$, so  $\QuotS{\Isom(\mathsf{R})}{\pm\rm{Id}}\subseteq\Aut(M(\mathsf{R}))$. 
Denote also by $G_a$ the subgroup of $W(\mathsf{R})$ of order $2$ containing the antipodal involution (if any), 
and let ${W^{\sigma}(\mathsf{R})}$ be the extension of ${W(\mathsf{R})}$ defined in Section~\ref{isom}.    

\begin{theorem}\label{ClassifResult}
Let $\mathsf{R}$ be a root system. 

\begin{itemize}

\item[(i)] If $\mathsf{R}$ is irreducible then the groups $\Aut(M(\mathsf{R}))$ are given in Table~\ref{answer}.
\begin{table}[!h]
\begin{center}
\caption{Automorphism groups of root system matroids}
\label{answer}
\begin{tabular}{c|c|c|c|c}
$\mathsf{R}$  & $|\mathsf{R}|$ & $|W(\mathsf{R})|$ & ${\Isom}(\mathsf{R})$&$\Aut(M(\mathsf{R}))$\\
\hline
$\mathsf{A}_n$&$n(n+1)$ & $(n+1)!$ &$W(\mathsf{A}_n)\times\ZZ_2$& $W(\mathsf{A}_n)$\\
$\mathsf{B}_n$&$2n^2$ & $2^{n} n!$ &$W(\mathsf{B}_n)$& $\QuotS{W(\mathsf{B}_n)}{G_a}$\\
$\mathsf{D}_4$&$24$ & $192$&$W(\mathsf{F}_4)$&$\QuotS{W(\mathsf{F}_4)}{G_a}$\\
$\mathsf{D}_n\; (n\geq 5)$&$2n(n-1)$ & $2^{n-1} n!$&$W(\mathsf{B}_n)$&$\QuotS{W(\mathsf{B}_n)}{G_a}$\\
$\mathsf{E}_6$&$72$ &$51840$ &$W(\mathsf{E}_6)\times\ZZ_2$&  $W(\mathsf{E}_6)$\\
$\mathsf{E}_7$&$126$ &$2903040$ &$W(\mathsf{E}_7)$& $\QuotS{W(\mathsf{E}_7)}{G_a}$\\
$\mathsf{E}_8$&$240$ &$696729600$ &$W(\mathsf{E}_8)$& $\QuotS{W(\mathsf{E}_8)}{G_a}$\\
$\mathsf{F}_4$& $48$  & $1152$ &$W(\mathsf{F}_4)$& $\QuotS{W^{\sigma}(\mathsf{F}_4)}{G_a}$\\
$\mathsf{H}_3$& $30$  & $120$  &$W(\mathsf{H}_3)$& $\QuotS{W^{\sigma}(\mathsf{H}_3)}{G_a}$\\
$\mathsf{H}_4$& $120$ & $14400$ &$W(\mathsf{H}_4)$& $\QuotS{W^{\sigma}(\mathsf{H}_4)}{G_a}$\\
$\mathsf{I}_2(m)$ & $2m$ & $2m$ &$W(\mathsf{I}_2(2m))$& $\Sym(m)$\\
\end{tabular}
\end{center}
\end{table}
\item[(ii)] If $\mathsf{R}=\sum_{i=1}^m p_i \mathsf{R}_i$ with $\mathsf{R}_i$
irreducible then
$$\Aut(M(\mathsf{R}))=\Pi_{i=1}^m \mathrm{wr}(\Sym(p_i), \Aut(M(\mathsf{R}_i)))$$
where $\mathrm{wr}$ stands for wreath product.


\end{itemize}


\end{theorem}

In Section~\ref{isom} we introduce coordinates for the root systems and describe additional symmetries. 
In Section~\ref{ex} we provide the proof of Theorems~\ref{m3} and~\ref{ClassifResult} for all the exceptional root systems.
Section~\ref{abd} is devoted to the proof of Theorem~\ref{m3} for classical root systems.  

\begin{remark}
It is worth to mention that the notion of root system can be extended
to any finitely generated Coxeter group (see~\cite[Section~5.4]{humphreyscoxeter}).
It would be interesting to see if Theorem~\ref{m3} holds in such a setting
with the corresponding extension of the notion of matroid to infinite sets.
\end{remark}

\section{Isometries and automorphism groups of root systems}
\label{isom}
We use standard coordinates for root systems of simple Lie algebras (except $G_2$), see~\cite[Section 2.10]{humphreyscoxeter}.
  
The root system $\mathsf{A}_n$ is the set of
roots $\{e_i-e_j\}$, $1\le i,j\le n+1$, in $\RR^{n+1}$.
All the roots are contained in
$n$-dimensional subspace with sum of coordinates equal
to zero. The group $W(\mathsf{A}_n)$ is the group ${\Sym}(n+1)$.
It is easy to see that $W(\mathsf{A}_n)$ does not contain an antipodal involution.
The group $\Isom(\mathsf{A}_n)$ is a central extension of  $W(\mathsf{A}_n)$ by the antipodal map. 
 
The root systems $\mathsf{E}_6$, $\mathsf{E}_7$ and $\mathsf{E}_8$ 
have special coordinates and are defined in~\cite{bourbaki,humphreyscoxeter}.
The groups $\Isom(\mathsf{E}_7)$ and  $\Isom(\mathsf{E}_8)$ coincide with  $W(\mathsf{E}_7)$ and  
$W(\mathsf{E}_8)$ respectively since the reflection groups already contain an antipodal involution.  
The group $\Isom(\mathsf{E}_6)$ is an extension of  $W(\mathsf{E}_6)$ by the antipodal map. 

The set $\mathsf{D}_n$ of roots is $\{\pm e_i\pm e_j\}$, $1\le i<j\le n$, in $\RR^{n}$. 
We describe its isometry group below. 

All root systems considered so far had only one orbit of roots under
$W(\mathsf{R})$ and so only one length of roots. The following root systems of 
simple Lie algebras have roots of two different lengths.  

The root system $\mathsf{B}_n$ is formed by the
roots $(\pm e_i)_{1\leq i\leq n}$ called {\em short roots}
and the roots $(\pm e_i\pm e_j)_{1\leq i<j\leq n}$ of
$\mathsf{D}_n$ called {\em long roots}. 
Note that $W(\mathsf{B}_n)$ preserves $\mathsf{D}_n$ as well.
This implies that $\Isom(\mathsf{D}_n)=\Isom(\mathsf{B}_n)=W(\mathsf{B}_n)$.

Denote by $\mathsf{D}'_{4}$ the root system formed by the $8$ vectors
$\pm e_i$ for $1\leq i\leq 4$ and the $16$ vectors
$(\pm 1/2, \pm 1/2, \pm 1/2, \pm 1/2)$. It is easy to see that $\mathsf{D}'_{4}$
is isomorphic to $\mathsf{D}_{4}$.
The root system $\mathsf{F}_4$ is the union of $\mathsf{D}_4$ and $\mathsf{D}'_{4}$.
The isometry group of $\mathsf{F}_4$ coincides with $W(\mathsf{F}_4)$. However, 
there is an isometry $\sigma$ of $\RR^4$ exchanging $\mathsf{D}_4$ with
${\sqrt{2}}\mathsf{D}'_{4}$. We denote by $W^{\sigma}(\mathsf{F}_4)$
the group generated by $W(\mathsf{F}_4)$ and $\sigma$. The group $W^{\sigma}(\mathsf{F}_4)$
does not preserve  $\mathsf{F}_4$, but it preserves the tessellation of $\RR^4$ by fundamental 
chambers of $W(\mathsf{F}_4)$.

\medskip
Now let us describe the root systems not corresponding to Lie algebras. 
Here we use the vectors of unit length only. For $\mathsf{I}_2(m)$ we assume $m\ge 5$ 
to exclude $\mathsf{B}_2$.

The root systems $\mathsf{H}_3$ and $\mathsf{H}_4$, see \cite{coxeter}, 
have coordinates in $\QQ(\sqrt{5})$. As a consequence,
the Galois involution $\sigma:\sqrt{5}\mapsto -\sqrt{5}$
can transform them into another root system, which is
isomorphic to the original one.
This involution exchanges the pair of roots of angle
$\arccos (\pm \frac{1+\sqrt{5}}{4})$ with the pair of roots of angle
$\arccos (\pm \frac{1-\sqrt{5}}{4})$.
We denote by $W^{\sigma}(\mathsf{H}_i)$ the group of
permutations of $\mathsf{H}_i$ generated by $W(\mathsf{H}_i)$
and the permutation induced by $\sigma$. As in the case of $\mathsf{F}_4$,
$W^{\sigma}(\mathsf{H}_i)$ preserves the $W(\mathsf{H}_i)$-action on $\RR^i$.

The root system $\mathsf{I}_2(m)$ for $m\geq 5$ is formed by the
$2m$ vectors $(\cos(\frac{\pi k}{m}), \sin(\frac{\pi k}{m}))_{1\leq k\leq 2m}$.
The roots form a regular $2m$-gon, so the group $\Isom(\mathsf{I}_2(m))$ is a dihedral group
isomorphic to $\mathsf{I}_2(2m)$. 

\medskip

As we have mentioned above, all isometries of $\mathsf{R}$ induce automorphisms of $M({\mathsf R})$.
To compute the groups $\Aut({\mathsf R})$ we need to list all the automorphisms of $M({\mathsf R})$
not induced by isometries of root systems.

\section{Exceptional root systems}
\label{ex}
In this section, we prove Theorems~\ref{m3} and~\ref{ClassifResult} for exceptional root systems.
 
Notice that $M({\mathsf R})$ is isomorphic to
$M({\mathsf R}')$ if and only if ${\mathsf R}$ and ${\mathsf R}'$ are
isomorphic as root systems. Indeed, if $c$ is a circuit of ${\mathsf{R}}$, then $c$
is contained in an irreducible root system. Thus, the question is
reduced to the irreducible case for which the classification gives
the answer by simply noticing that isomorphism preserves the dimension and
the number of elements.

\begin{proof}[Proof of the theorems for ${\mathsf R}=\mathsf{E}_6,\mathsf{E}_7,\mathsf{E}_8,\mathsf{F}_4,\mathsf{H}_3,\mathsf{H}_4,\mathsf{I}_2(m)$] 
First, we explain the assertion (ii) of Theorem~\ref{ClassifResult}. If
$\mathsf{R}=\sum_{i=1}^m p_i \mathsf{R}_i$ and $\phi\in \Aut(M(\mathsf{R}))$,
then $\phi$ permutes the components isomorphic to $\mathsf{R}_i$ and
thus $\phi$ belongs to the mentioned product of wreath products.

Now consider irreducible root systems. We prove the theorems using case by case analysis.
For $\mathsf{I}_2(m)$  it is clear that any two non-antipodal
roots form a basis and thus $\Aut(M(\mathsf{I}_2(m)))=\Sym(m)$.
All the circuits are of order $3$, so Theorem~\ref{m3} holds as well.

\medskip

If ${\mathcal F}$ is a family of subsets of $X$,
denote by $G(X, {\mathcal F})$  the graph on
$|X|+|{\mathcal F}|$
vertices with vertex $x\in X$ being adjacent
to $S\in {\mathcal F}$ if and only if $x\in S$.
The group $\Aut(G(X, {\mathcal F}))$ of automorphisms
of the graph $G(X, {\mathcal F})$ is identified
with a subgroup of the symmetric group $\Sym(X)$.
The program {\tt nauty} \cite{nauty} can compute the automorphism
group of a graph $G$. Moreover, if one attributes colors to vertices
then this program can compute the group of automorphism preserving
those colors.

If ${\mathcal F}$ is a family of subsets of $X$ invariant
under the automorphism group $\Aut(M)$ of a matroid $M$
on $X$, then
$\Aut(M)\subset \Aut(G(X, {\mathcal F}))$.
If one takes ${\mathcal F}={\mathcal I}$,
${\mathcal B}$ or ${\mathcal C}$, then we have equality.
Take $\mathsf{R}$ an irreducible root system.
If we can check that all elements of $\Aut(G(X, {\mathcal C}_3))$
are actually symmetries of $M(\mathsf{R})$ then we have 
$\Aut(M(\mathsf{R}))=\Aut(G(X, {\mathcal C}_3))$
and proved Theorem~\ref{m3} for $\mathsf{R}$. At the same time, 
the set $\mathcal{C}_3$ is not large for the exceptional root systems, 
and it can be easily computed, as well as the group $\Aut(G(X, {\mathcal C}_3))$.  
This method works directly for the root systems $\mathsf{E}_6$, $\mathsf{E}_7$, $\mathsf{E}_8$, 
$\mathsf{F}_4$, $\mathsf{H}_3$ and $\mathsf{H}_4$\footnote{All sources of the programs of 
this paper are available at \url{http://www.liga.ens.fr/~dutour/RootMatroid/}}.

In particular, we compute the automorphism groups of the root systems matroids themselves. 
The results are listed in Table~\ref{answer}, which completes the proof of Theorem~\ref{ClassifResult}.  

\end{proof}

\section{Proof of the $\mathsf{A}_n$, $\mathsf{D}_n$, $\mathsf{B}_n$ cases}
\label{abd}

In~\cite{FT} the correspondence between circuits in root systems of 
simple Lie algebras and 
Euclidean simplices generating discrete reflection groups is described.

Any circuit in a root system defines
(up to similarity) a Euclidean simplex generating
a discrete reflection group.
Given $(n+1)$-tuple of roots, we take $n+1$ hyperplanes
orthogonal to these roots and passing through the origin.
Now choose any of the hyperplanes and translate it
in such a way that the image does not contain the origin.
The new hyperplane together with the remaining $n$ ones
define a simplex.

Conversely, any Euclidean simplex generating a discrete reflection
group defines a circuit in some root system in the
following way: faces of codimension one are orthogonal to roots of 
some affine root system. These roots define a circuit of the underlying
finite root system.

As a consequence, using the classification of simplices generating discrete 
reflection groups obtained in~\cite{FT}, we get the circuits of the
root systems $\mathsf{A}_n$, $\mathsf{B}_n$, and
$\mathsf{D}_n$.

Denote by $V(\mathsf{R})$ the set of pairs of opposite roots of $\mathsf{R}$.
 If $J$ is a subset of $V(\mathsf{R})$, we say that a
root $v\in \mathsf{R}$ belongs to $J$ if $J$ contains
a vertex $(v,-v)\in V(\mathsf{R})$.
In fact, there is no difference in defining $J$ in terms
of roots or pairs of opposite roots.
We will use pairs sometimes to emphasize that we are
able to choose any representative from a pair.

To prove Theorem~\ref{m3} it is
sufficient to show that if the set of circuits of
$M(\mathsf{R})$ of order $3$ is invariant under some
permutation of elements of $V(\mathsf{R})$, then the
set of all circuits of $M(\mathsf{R})$ is invariant, too.
The latter is an immediate corollary of the following
lemma.

\begin{lemma}\label{induction}
Let $k\ge 3$ be an integer not exceeding $n$, and let
$f$ be a permutation of elements of $V(\mathsf{R})$.
If the set of circuits of $M(\mathsf{R})$ of order
not exceeding $k$ is invariant under $f$, then so is the
set of circuits of order $k+1$.
\end{lemma}
\begin{proof} We prove the lemma for root systems $\mathsf{A}_n$,
$\mathsf{D}_n$, and $\mathsf{B}_n$ separately:

{\bf Case} $\mathsf{R}=\mathsf{A}_n$.\\
For any set $J$ of elements of $V(\mathsf{R})$ we draw
the following graph $\Gamma(J)$:
\begin{itemize}
\item the vertices of $\Gamma(J)$ are $e_i$ for
those $i$ which take part in the expression of at least
one element of $J$;
\item two vertices $e_i$ and $e_j$ are joined by
an edge if the root $\pm (e_i-e_j)\in J$.
\end{itemize}
As it is shown in~\cite{FT}, graphs $\Gamma(J)$
corresponding to circuits $J$ of $M(\mathsf{A}_n)$
are cycles, and conversely.
The order of $J$ is equal to the number
of edges in $\Gamma(J)$.
Take any circuit $J$ of order $k+1$.
It is sufficient to prove that $f(J)$ is linearly dependent. 

Since $\Gamma(J)$ is a cycle, we may assume that $J=\{e_1-e_2, e_2-e_3,\dots, e_k-e_{k+1}, e_1-e_{k+1}\}$.
Now consider $J'=J\cup \{e_1-e_3\}$.
The graph $\Gamma(J')$ consists of cycles
$\{e_1-e_2, e_2-e_3, e_1-e_3\}$ and
$\{e_1-e_3, e_3-e_4,\dots, e_k-e_{k+1}, e_1-e_{k+1}\}$.
By the assumption of the lemma, the images of 
these circuits under $f$ are circuits.
Therefore, from one cycle we may express
$f(e_1-e_3)$ as a linear combination of $f(e_1-e_2)$
and $f(e_2-e_3)$, and from another cycle we may express
$f(e_1-e_3)$ as a linear combination
of $f(e_3-e_4),\dots,f(e_k-e_{k+1}), f(e_1-e_{k+1})$.
Subtracting one expression from another, we obtain
a dependence on vectors of $f(J)$.

{\bf Case} $\mathsf{R}=\mathsf{D}_n$.\\
For any set $J$ of elements of $V(\mathsf{R})$ we draw the
following edge colored graph $\Gamma(J)$:
\begin{itemize}
\item the vertices are $e_i$ for those $i$ which
take part in the expression of at least one element of $J$;
\item two vertices $e_i$ and $e_j$ are joined by
a {\it red} edge if the root $\pm(e_i+e_j)\in J$;
\item two vertices $e_i$ and $e_j$ are joined by
a {\it black} edge if the root $\pm(e_i-e_j)\in J$.
\end{itemize}

According to~\cite{FT}, circuits correspond either
to
\begin{itemize}
\item cycles with even number of red edges,
\item or to two cycles $C_1$, $C_2$, possibly of length $2$, joined by a path,
such that the number of red edges in each cycle is odd (edges of the path are 
colored in any way). 

\end{itemize}
We take any circuit $J$ of order $k+1$ and show that
$f(J)$ is linearly dependent.
If one of the cycles of $\Gamma(J)$ contains at
least $4$ vertices, we do almost the same procedure
as in the $\mathsf{A}_n$ case.
Suppose the cycle contains vertices $e_1,e_2,e_3$, and $e_4$.
Consider $J'$ obtained from $J$ by adding either
$e_1-e_3$ or $e_1+e_3$ so that the number
of red edges in the new cycle of order $3$ is even.
Then we obtain two new circuits of order at most $k$,
and they intersect by a unique root.
A reasoning similar to the one for $\mathsf{A}_n$
completes the proof of this case.

So, we may assume that all cycles have length at most $3$.
Since $k+1\ge 4$, $\Gamma(J)$ contains two cycles.
Now, take two vertices of valency two belonging to
distinct cycles of $\Gamma(J)$, and join them by an edge.
Choose the color of the edge in such a way that one
of the shortest cycles containing this edge contains an
even number of red edges.
Clearly, this cycle corresponds to a circuit of
order not exceeding $k$.
If we hide the edges of this cycle belonging to two
initial cycles of $\Gamma(J)$, we obtain another
circuit of $M(\mathsf{R})$ of order not exceeding $k$.
By the assumption, the images of corresponding sets
of roots under $f$ are circuits again.
By minimality, the corresponding linear dependencies
contains all the roots with non-zero coefficients.
Eliminating the new root from two linear dependencies,
we obtain that $f(J)$ is linearly dependent.   

{\bf Case} $\mathsf{R}=\mathsf{B}_n$.\\
For any set $J$ of elements of $V(\mathsf{R})$ we draw the
following graph $\Gamma(J)$ with colored edges and
some vertices marked:
\begin{itemize}
\item the vertices are $e_i$ for those $i$ which
take part in the expression of at least one element of $J$;
\item two vertices $e_i$ and $e_j$ are joined by
a {\it red} edge if the root $\pm(e_i+e_j)\in J$;
\item two vertices $e_i$ and $e_j$ are joined by
a {\it black} edge if the root $\pm(e_i-e_j)\in J$;
\item a vertex $e_i$ is {\it marked} if the root
$\pm e_i\in J$.
\end{itemize}
By~\cite{FT}, circuits correspond either to
\begin{itemize}
\item the graphs of $\mathsf{D}_n$ case,
\item or a path with two end vertices being marked,
\item or to a cycle with an odd number of red edges and
a path linking the cycle to a unique marked vertex.
\end{itemize}
The order of $J$ is the number of edges of $\Gamma(J)$
plus the number of marked vertices.
Again, we consider any circuit $J$ of order $k+1$
and show that $f(J)$ is linearly dependent.  

If $\Gamma(J)$ does not contain any marked vertices,
then $J$ contains long roots only, so it belongs
to $\mathsf{D}_n$ and the proof repeats the above one.
Thus, we may assume that we have at least one
marked vertex.

Suppose that $\Gamma(J)$ contains a path to a marked vertex $e_1$.
Let the neighboring vertex be $e_2$.
Consider $J'=J\cup \{e_2\}$.
Then $J'$ consists of two circuits of $M(\mathsf{R})$,
namely of circuit of order $3$ containing $e_1$,
$e_2$ and the edge joining these two vertices, and
the remaining elements of $J$ together with $e_2$.
Both circuits have order at most $k$.
By the same method we see that $f(J)$ is linearly
dependent.

We are left with the case when $\Gamma(J)$ is a cycle
with one vertex marked. If there are at least $4$
vertices in $\Gamma(J)$, we take two non-neighboring
non-marked vertices, join them by an edge of an
appropriate color and use the same method as in the
$\mathsf{D}_n$ case.
So, the only interesting case is when $J$ is of order $4$.
We may assume that $J=\{e_1, e_1\pm e_2,e_2\pm e_3,
e_1\pm e_3\}$ and that the number of plus signs occurring
in its expression is odd.
The set $J'=J\cup \{e_2\}$ contains the following two circuits
of $M(\mathsf{R})$: $\{e_1,e_1\pm e_2,e_2\}$, and
$\{e_2,e_2\pm e_3,e_1\pm e_3,e_1\}$.
Since they both have paths to marked vertices, 
their images under $f$ are still circuits; so we have
two linear dependencies on the images.
Eliminating $e_2$ from them, we obtain a dependence
on $f(J)$. The result is not trivial since one of the
dependencies contains $f(e_1\pm e_2)$ with non-zero
coefficient, and the other one contains images of two
long roots with non-zero coefficients. 

\end{proof}


\begin{thebibliography}{99}






\bibitem{bourbaki}
N. Bourbaki, {\em Groupes et alg\`ebres de Lie, Chapitres \rm{IV}--{VI}}, Hermann, Paris, 1968.  

\bibitem{coxeter}
H.S.M. Coxeter, {\em Regular polytopes}, Dover Publications, New York, 1973.






\bibitem{FT}
A.~Felikson, P.~Tumarkin, {\em Euclidean simplices generating discrete reflection groups},
European J. Combin. {\bf 28} (2007) 1056--1067. 

\bibitem{cubesimplexcases}
L. Fern, G. Gordon, J. Leasure and S. Pronchik, {\em Matroid Automorphisms and Symmetry Groups}, Combinatorics, Probability and Computing {\bf 9} (2000) 105--123.


\bibitem{F4syst}
S. Fried, A. Gerek, G. Gordon and A. Peruni\u ci\'c,
{\em Matroid automorphisms of the $F_4$ root system}, 
Electronic journal of combinatorics {\bf 14} (2007) R78.

\bibitem{ehly}
K. Ehly, G. Gordon, {\em Matroid automorphisms of the root system $H_3$},
Geom. Dedicata {\bf 130} (2007) 149--161.


\bibitem{humphreyscoxeter}
J.E. Humphreys, {\em Reflection groups and Coxeter groups},
Cambridge University Press, 1990.

\bibitem{nauty}
B.D. McKay, {\em The nauty program}, \url{http://cs.anu.edu.au/people/bdm/nauty/}


\end{thebibliography}
\end{document}